\documentclass[12pt,a4paper]{amsart}

\usepackage{amsfonts, amsmath, amssymb, amsthm, amscd, hyperref}

\usepackage{a4wide}

\newtheorem{thm}{Theorem}
\newtheorem*{thm*}{Theorem}
\newtheorem{lem}{Lemma}

\newtheorem{cor}[thm]{Corollary}

\theoremstyle{definition}
\newtheorem{defn}{Definition}

\theoremstyle{remark}

\newcommand{\pt}{\mathrm{pt}}

\DeclareMathOperator{\gfree}{g_{\rm free}}

\renewcommand{\int}{\mathop{\rm int}}

\renewcommand{\epsilon}{\varepsilon}

\begin{document}

\title{Tverberg-type theorems for intersecting by rays}

\author{R.N.~Karasev}
\thanks{This research is supported by the Dynasty Foundation, the President's of Russian Federation grant MK-113.2010.1, the Russian Foundation for Basic Research grants 10-01-00096 and 10-01-00139}

\email{r\_n\_karasev@mail.ru}
\address{
Roman Karasev, Dept. of Mathematics, Moscow Institute of Physics and Technology, Institutskiy per. 9, Dolgoprudny, Russia 141700}

\keywords{central point theorem, Tverberg's theorem, Helly's theorem}
\subjclass[2000]{52A20, 52A35, 52C35}

\begin{abstract}
In this paper we consider some results on intersection between rays and a given family of convex, compact sets. These results are similar to the center point theorem, and Tverberg's theorem on partitions of a point set.
\end{abstract}

\maketitle

\section{Introduction}

In this paper we consider some results on intersection between rays and a given family of convex, compact sets, that resemble the center point theorem of~\cite{neumann1945,rado1946}, and Tverberg's theorem on partitions from~\cite{tver1966}. 

Let us make a definition. Consider a straight line $\ell\subset\mathbb R^d$ and a point $p\in\ell$. The point $p$ divides $\ell$ into two half-lines, we call these half-lines \emph{rays starting at $p$}. We are going to study the questions of the following type: given a family $\mathcal F$ of convex sets in $\mathbb R^d$, find a point $p\in\mathbb R^d$ such that every ray starting at $p$ intersects at least $\alpha|\mathcal F|$ members of $\mathcal F$, or at most $\beta|\mathcal F|$ members of $\mathcal F$. Such questions were considered before in~\cite{ruhu1999, kar2008dcpt}, for the case of hyperplanes, and in~\cite{fhp2009,kar2009} for families of convex sets.

The following theorem is similar to the ``dual'' Tverberg theorem for hyperplanes from~\cite{kar2008dcpt}, the statements of this kind (with minor differences) for hyperplanes were conjectured in~\cite{ruhu1999}.

\begin{thm}
\label{tverberg-rays}
Let $\mathcal F$ be a family of $n$ compact convex sets in $\mathbb R^d$, such that any point $x\in\mathbb R^d$ belongs to at most $c$ sets of $\mathcal F$. Suppose that $r$ is a prime power and the following inequality holds
$$
n\ge (d+1)(r-1)+c+1.
$$
Then $\mathcal F$ has $r$ disjoint subfamilies $\mathcal F_1, \ldots, \mathcal F_r$, such that there exists a point $p\in\mathbb R^d$ with the following property: for any ray $\rho$ starting at $p$, and any subfamily $\mathcal F_i$, there exists $K\in\mathcal F_i$ such that $\rho\cap K=\emptyset$.
\end{thm}  

The following theorem is a generalization of the result of~\cite{fhp2009}, see also~\cite{ruhu1999}, where a particular case was conjectured for families of hyperplanes. This is an analogue of the central point theorem for finite point sets, see~\cite{neumann1945,rado1946,grun1960}.

\begin{cor}
\label{cpt}
Let $\mathcal F$ be a family of $n$ compact convex sets in $\mathbb R^d$, such that any point $x\in\mathbb R^d$ belongs to at most $c$ sets of $\mathcal F$. Suppose that $r$ is a positive integer and the following inequality holds
$$
n\ge (d+1)(r-1)+c+1.
$$
Then there exists a point $p\in\mathbb R^d$ such that any ray $\rho$ starting at $p$ does not intersect at least $r$ of the sets in $\mathcal F$.
\end{cor}

Theorem~\ref{tverberg-rays} is formulated for compact sets, and the compactness is essential in the proof. Still, it is possible to formulate a similar result for hyperplanes. Let us make some definitions.

\begin{defn}
A convex open subset $G\subset\mathbb R^d$ is called \emph{almost bounded}, if it does not contain an open cone. Equivalently, for any point $p\in G$ the set of rays starting at $p$, and lying within $G$, has an empty interior as a subset of the unit sphere $S^{n-1}$.
\end{defn}

\begin{defn}
For a family of hyperplanes $\mathcal G$ in $\mathbb R^d$ denote by $C(\mathcal G)$ the union of all almost bounded components of the complement $\mathbb R^d\setminus \bigcup\mathcal G$.
\end{defn}

The following theorem generalizes the dual Tverberg theorem from~\cite{kar2008dcpt} to the case, when hyperplanes are not in general position. This statement is also a partial solution of Conjecture~2 in~\cite{ruhu1999}.

\begin{thm}
\label{tverberg-hp}
Let $\mathcal F$ be a family of $n$ hyperplanes in $\mathbb R^d$, such that any point $x\in\mathbb R^d$ belongs to at most $c$ hyperplanes of $\mathcal F$. Suppose that $r$ is a prime power and the following inequality holds
$$
n\ge (d+1)(r-1)+c+1.
$$
Then $\mathcal F$ has $r$ disjoint subfamilies $\mathcal F_1, \ldots, \mathcal F_r$, such that
$$
\bigcap_{i=1}^r C(\mathcal F_i)\not=\emptyset. 
$$
\end{thm}

The proofs in this paper mostly follow the proofs in~\cite{kar2008dcpt}, the essential difference is that the general position requirements are substituted by an upper bound of the covering multiplicity of a family. Such strengthening is allowed by an accurate use of the concept of the Krasnosel'skii-Schwarz genus (see~Section~\ref{genus-sec} for the definition) to avoid singular configurations that give a solution of the topological problem (in terms of sections of a vector bundle), but do not correspond to the solution of the original geometric problem.

\section{Facts from topology}
\label{topology}

In this section some topological facts, that arise in the proof of Theorem~\ref{tverberg-rays} are given. In fact, the first part of the proof follows the proof of Theorem~1.1 in~\cite{kar2009}, this and the following sections restate the needed lemmas.

We consider topological spaces with continuous (left) action of a finite group $G$ and continuous maps between such spaces that commute with the action of $G$. We call them $G$-spaces and $G$-maps. In this paper we actually consider groups $G=(Z_p)^k$ for prime $p$, called usually \emph{$p$-tori}, but most of the definitions are valid for arbitrary finite group $G$.

For basic facts about (equivariant) topology and vector bundles the reader is referred to the books~\cite{hsiang1975,mishch1998,milsta1974}. The cohomology is taken with coefficients $Z_p$ ($p$ is the same as in the definition of $G$), in notations we omit the coefficients. Let us start from some standard definitions.

\begin{defn}
Denote by $EG$ the classifying $G$-space, which can be thought of as an infinite join $EG=G*\dots *G*\dots$ with diagonal left $G$-action. Denote $BG=EG/G$. For any $G$-space $X$ denote $X_G=(X\times EG)/G$, and put (\emph{equivariant cohomology in the sense of Borel}) $H_G^*(X) = H^*(X_G)$. It is easy to verify that for a free $G$-space $X$, the space $X_G$ is homotopy equivalent to $X/G$. 
\end{defn}

Consider the algebra of $G$-equivariant cohomology of the point $A_G = H_G^*(\pt) = H^*(BG)$. For a group $G=(Z_p)^k$, the algebra $A_G=H_G^*(Z_p)$ has the following structure (see~\cite{hsiang1975}). In the case $p$ odd it has $2k$ multiplicative generators $v_i,u_i$ with dimensions $\dim v_i = 1$ and $\dim u_i = 2$ and relations
$$
v_i^2 = 0,\quad\beta{v_i} = u_i.
$$
We denote by $\beta(x)$ the Bockstein homomorphism. 

In the case $p=2$ the algebra $A_G$ is the algebra of polynomials of $k$ one-dimensional generators $v_i$.

Any representation of $G$ can be considered as a vector bundle over the point $\pt$, and it has corresponding characteristic classes in $H_G^*(\pt)$. We need the following lemma, that follows from the results of~\cite{hsiang1975}, Chapter III \S 1.

\begin{lem}
\label{euler-nz}
Let $G=(Z_p)^k$, and let $I[G]$ be the subspace of the group algebra $\mathbb R[G]$, consisting of elements
$$
\sum_{g\in G} a_g g,\quad \sum_{g\in G} a_g = 0.
$$
Then the Euler class $e(I[G])\not=0\in A_G$ and is not a divisor of zero in $A_G$. 
\end{lem}

Note that in this lemma the fact that $G=(Z_p)^k$ is essential.

\section{Topology of Tverberg's theorem}

This paper reproduces some lemmas from~\cite{kar2009}. In Tverberg's theorem and its topological generalizations (see~\cite{bss1981,vol1996} for example) it is important to consider the configuration space of $r$-tuples of points $x_1,\ldots, x_r\in \Delta^N$ with pairwise disjoint supports. Here $\Delta^N$ is a simplex of dimension $N$. Let us make some definitions, following the book~\cite{mat2003}.

\begin{defn}
Let $K$ be a simplicial complex. Denote by $K_{\Delta}^r$ the subset of the $r$-fold product $K^r$, consisting of the $r$-tuples $(x_1,\ldots, x_r)$ such that every pair $x_i, x_j$ ($i\not=j$) has disjoint supports in $K$. We call $K_{\Delta}^r$ the \emph{$r$-fold deleted product of $K$}.
\end{defn}

\begin{defn}
Let $K$ be a simplicial complex. Denote by $K_{\Delta}^{*r}$ the subset of the $r$-fold join $K^{*r}$, consisting of convex combinations $w_1x_1\oplus\dots\oplus w_rx_r$ such that every pair $x_i, x_j$ ($i\not=j$) with weights $w_i,w_j>0$ has disjoint supports in $K$. We call $K_{\Delta}^{*r}$ the \emph{$r$-fold deleted join of $K$}.
\end{defn}

Note that the deleted join is a simplicial complex again, while the deleted product has no natural simplicial complex structure, although it has some cellular complex structure.

The $r$-fold deleted product of the simplex $\Delta^{(r-1)(d+1)}$ is the natural configuration space in Tverberg's theorem, but sometimes it is simpler to use the deleted join. Denote by $[r]$ the set $\{1, \ldots, r\}$, with the discrete topology.

If $r$ is a prime power $r=p^k$, then the group $G=(Z_p)^k$ can be somehow identified with $[r]$, so a $G$-action on $K_\Delta^r$ and $K_\Delta^{*r}$ by permuting $[r]$ arises. The following lemma is well-known, see~\cite{vol1996} for example.

\begin{lem}
\label{del-join-ind}
The deleted join of the simplex $(\Delta^N)_\Delta^{*r} = [r]^{*N+1}$ is $N-1$-connected, and the natural map $A_G^l\to H_G^l((\Delta^N)_\Delta^{*r})$ is injective for $l \le N$.   
\end{lem}

Let us say a few words about the proof. There is the Leray-Serre spectral sequence that relates the ordinary cohomology of a $G$-space $X$ to its equivariant cohomology, the bottom row of $E_2$ in this spectral sequence being $A_G^*$. The connectedness hypothesis implies that the corresponding part of the bottom row survives in $E_\infty$, that is the statement of the lemma.

The next lemma is used in~\cite{vol1996} too, a proof of this lemma can be found in~\cite{kar2009}, for example.

\begin{lem}
\label{del-prod-ind}
Let $r=p^k$, $G=(Z_p)^k$, and let $K$ be a simplicial complex. If the natural map $A_G^l\to H_G^l(K_\Delta^{*r})$ is injective for $l\le N$, then the similar map $A_G^l\to H_G^l(K_\Delta^r)$ is injective for $l \le N - r + 1$.  
\end{lem}

\section{The genus of $G$-spaces}
\label{genus-sec}

In this section we describe some measure of complexity for a $G$-space. Let $X$ be a paracompact free $G$-space, $G$ being a finite group. Informally, the main idea is that this measure can be estimated from the equivariant cohomology of $X$, by the statements like those in Lemmas~\ref{del-join-ind} and \ref{del-prod-ind}. Let us make a definition.

\begin{defn}
\emph{The free genus} of a free $G$-space $X$ is the least number $n$ such that $X$ can be covered by $n$ open subsets $X_1,\ldots, X_n$ so that every $X_i$ can be $G$-mapped to $G$. Denote the free genus by $\gfree(X)$.
\end{defn}

There are several kinds of genus for a $G$-space, here we only use the free genus, and call it simply ``genus''. The free genus was introduced in~\cite{kr1952,schw1957,schw1966}, different versions of this definition for non-free action are discussed in~\cite{bart1993}.

Let us explain the definition of the genus. The set $X_i$ in the definition can be $G$-mapped to $G$ iff the group $G$ acts on connected components of $X_i$ freely, we call such spaces \emph{inessential} in the sequel. In fact, for paracompact $X$ the sets $X_i$ in the definition of genus may be taken closed instead of open.

Let us state the properties of the genus, valid for paracompact spaces, following~\cite{vol2007}. 

\begin{enumerate}
\item (Monotonicity)
If there is a $G$-map $f: X\to Y$, then $\gfree(X) \le \gfree(Y)$;

\item (Subadditivity)
Let $X=A\cup B$, where $A$, $B$ are closed or open $G$-invariant subspaces. Then $\gfree(X) \le \gfree(A)+\gfree(B)$;

\item (Dimension upper bound)
$\gfree(X)\le \dim X+1$;

\item (Cohomology lower bound)
If the natural map $A_G^n\to H_G^n(X, M)$ is nonzero for some $G$-module $M$, then $\gfree(X) \ge n+1$.
\end{enumerate}

Take the deleted join $(\Delta^N)_\Delta^{*r}$ and the deleted product $(\Delta^N)_\Delta^r$, considered in the previous section for $r$ being a prime power, with an action of the corresponding $p$-torus. Then the cohomology lower bound and the dimension upper bound, with Lemmas~\ref{del-join-ind} and~\ref{del-prod-ind} give
$$
\gfree((\Delta^N)_\Delta^{*r}) = N+1,\quad \gfree((\Delta^N)_\Delta^r) = N-r+2.
$$

We need the following lemma, that can be considered a strengthening of the definition of genus. A particular case of this lemma for $G=Z_2$ was proved in~\cite[Theorem~9]{kar2009bu}.

\begin{lem}
\label{covering-genus}
Let $X$ be a paracompact $G$-space, let $\mathcal U = \{U_i\}_{i=1}^N$ be some open (or closed) covering of $X$ by inessential invariant subsets. Then there exist a point $x\in X$, that is covered by at least $\gfree(X)$ sets of $\mathcal U$.
\end{lem}

\begin{proof}
Since every $U_i$ can be mapped to $G$, then from the partition of unity, corresponding to $\mathcal U$, arises a map $f : X \to G^{*N}$. 

Consider the contrary: the covering $\mathcal U$ has multiplicity at most $\gfree(X)-1$. Then the image of $f$ is within the $\left(\gfree(X)-2\right)$-dimensional skeleton of $G^{*N}$. Now from the dimension upper bound and the monotonicity of the genus it follows that $\gfree(X)\le \gfree(X)-1$, which is a contradiction. 
\end{proof}

Note that this lemma is true if we consider the fixed-point-free genus $g_G(X)$ (see~\cite{bart1993,vol2007}) of a fixed point free $G$-space, and call a subset \emph{inessential} if none of its connected components is stabilized by the whole group $G$. This follows from the dimension upper bound for fixed-point-free genus.

\section{Proof of Theorem~\ref{tverberg-rays}}

Consider the simplex $\Delta = \Delta^{n-1}$, along with some identification of its vertices with $\mathcal F$. Take some large enough ball $B\subset\mathbb R^d$, containing all the sets of $\mathcal F$ in its interior. The configuration space that we study is $\Delta^r_\Delta\times B$, denote its elements by $(\alpha_1, \alpha_2, \ldots, \alpha_r, p)$. The points $\alpha_i$ in the simplex $\Delta$ will be considered as functions $\alpha_i : \mathcal F\to \mathbb R^+$ with unit sum.

Denote for brevity $\mathbb R^d = V$. Now let us map our configuration space to $V^r$ by the following rule. Let $\pi_K(p)$ be the orthogonal projection of $p$ to $K\in\mathcal F$. Put 
$$
f(\alpha_1, \alpha_2, \ldots, \alpha_r, p) = \bigoplus_{i=1}^r \sum_{K\in\mathcal F} \alpha_i(K) (\pi_K(p) - p),
$$ 
This map is evidently continuous and $G$-equivariant, if we identify $V^r$ with $V[G]$ ($V$-valued functions on $G$ with $G$-action by right multiplication by $g^{-1}$).

Denote the zero set of $f$ by $Z$. Similar to~\cite{kar2009}, the map $f$ can be considered as a section of $G$-equivariant vector bundle, its Euler class being
$$
e(f) = w^d\times u\in H_G^{rd}(\Delta^r_\Delta\times B, \Delta^r_\Delta\times \partial B),
$$
where $w$ is the image of the $e(I[G])$, $u$ is the generator of $H^d(B, \partial B)$. By Lemmas~\ref{euler-nz} and \ref{del-prod-ind}, $w^d\not=0\in H_G^{d(r-1)}(\Delta^r_\Delta)$, and $e(f)\not=0$. 

Similar to the proof of Lemma~\ref{del-prod-ind} in~\cite{kar2009}, we conclude that the natural map $A_G^l\to H_G^l(Z)$ is injective in dimensions $l \le n - r - (r-1)d = n - 1 - (r-1)(d+1)$. Let us sketch the proof of this claim. Suppose that some $\xi\in A_G^l$ maps to zero in $H_G^l(Z)$ by the natural map $\pi_Z^* : A_G^*\to H_G^*(Z)$, then by the properties of the cohomology multiplication
$$
\xi w^d\times u = 0\in H_G^{rd + l}(\Delta^r_\Delta\times B, \Delta^r_\Delta\times \partial B),
$$ 
which contradicts with Lemma~\ref{del-prod-ind}. 

It follows from the cohomology lower bound on the genus that $\gfree(Z)\ge n - (r-1)(d+1) \ge c + 1$. Now we are going to use this fact and show that the point $p$ is not contained in any $K\in\mathcal F$ with $\alpha_i(K)>0$.

We can find small enough $\varepsilon>0$ so that the family of $\varepsilon$-neighborhoods $\mathcal F(\varepsilon) = \{K(\varepsilon)\}_{K\in\mathcal F}$ has covering multiplicity at most $c$. Now consider the following open subsets of $Z$: for any $K\in\mathcal F$ denote 
$$
U_K = \{(\alpha_1, \alpha_2, \ldots, \alpha_r, p)\in Z : \exists i\in[r]\ \text{such that}\ \alpha_i(K) > 0\ \text{and}\ p\in K(\varepsilon)\}.
$$
Note that for any $(\alpha_1, \alpha_2, \ldots, \alpha_r, p)\in U_K$ there is only one $i\in[r]$ such that $\alpha_i(K) > 0$, since we consider the deleted product $\Delta^r_\Delta$. Hence the set $U_K$ is partitioned into connected components, that are permuted by $G$ freely, i.e. it is inessential. The family $\{U_K\}$ covers $Z$ with multiplicity at most $c$. If it does cover $Z$, than $\gfree(Z)\le c$, that was shown above to be false. 

Therefore, there exists a combination $(\alpha_1, \alpha_2, \ldots, \alpha_r, p)$ with the following property: if $\alpha_i(K) > 0$, then $p\not\in K(\varepsilon)$. Put
$$
\mathcal F_i = \{K\in\mathcal F : \alpha_i(K) > 0\},
$$
the families $\mathcal F_i$ are disjoint. For any $i\in[r]$ the point $p$ is in the convex hull of the points $X_i = \{\pi_K(p)\}_{K\in\mathcal F_i}$, reducing the family $\mathcal F_i$ if needed, we may assume that $p$ is in the relative interior of $X_i$. It is clear, that for any ray $\rho$ starting at $p$, some of the angles $\angle(\rho, \pi_K(p)-p)$ ($K\in\mathcal F_i$) is at least $90^\circ$, and $\rho$ cannot intersect the corresponding set $K$.  

\section{Proof of Corollary~\ref{cpt}}

If $r$ is a prime power, then the statement follows from Theorem~\ref{tverberg-rays}. Otherwise choose a positive integer $k$ so that $R=k(r-1)+1$ is prime, such $k$ exists by the Dirichlet theorem on arithmetic progressions. Now consider the family $\mathcal G$ of $kn$ sets, that is obtained from $\mathcal F$ by taking each member of $\mathcal F$ exactly $k$ times. Any point in $\mathbb R^d$ belongs to at most $kc$ sets of $\mathcal G$. The inequality
$$
kn \ge (d+1)(R-1) + kc + 1 = k(d+1)(r-1)+kc+1
$$ 
holds since $kn \ge k(d+1)(r-1)+kc+k$. Hence there exists a point $p\in\mathbb R^d$ such that any ray $\rho$ starting at $p$ does not intersect at least $R$ members of $\mathcal G$, In this case it is clear that $\rho$ does not intersect at least $r$ members of $\mathcal F$.

\section{Proof of Theorem~\ref{tverberg-hp}}

The proof mainly follows the proof of Theorem~\ref{tverberg-rays}, though some changes are required. 

Denote again 
$$
f(\alpha_1, \alpha_2, \ldots, \alpha_r, p) = \bigoplus_{i=1}^r \sum_{K\in\mathcal F} \alpha_i(K) (\pi_K(p) - p),
$$ 
to use the above reasonings, the map $f$ should not have zeros on $\Delta^r_\Delta\times\partial B$ for large enough ball $B$. But in the case of hyperplanes this is not true. We need the following lemma from~\cite{adw1984}.

\begin{lem}
\label{projbounded}
Suppose $\mathcal F=\{h_1, \ldots, h_n\}$ is a set of hyperplanes in $\mathbb R^d$, consider the orthogonal projections $\pi_1, \ldots, \pi_n$ onto the respective hyperplanes. Then there exists a convex body $P$, such that 
$$
\forall i=1,\ldots,n,\ \pi_i(P)\subseteq P.
$$ 
\end{lem}

Take the convex body $P$ from Lemma~\ref{projbounded}. Denote the zero set of $f$ on $\Delta^r_\Delta\times P$ by $Z$, this set still can have nonempty intersection with $\Delta^r_\Delta\times\partial P$. 

Suppose that $P$ contains the origin, and approximate the map $f$ on $P$ by   
$$
f_\varepsilon(\alpha_1, \alpha_2, \ldots, \alpha_r, p) = \bigoplus_{i=1}^r \sum_{K\in\mathcal F} \alpha_i(K) ((1-\varepsilon)\pi_K(p) - p),
$$
denote its zero set by $Z_\varepsilon$. It is clear that 
$$
Z_\varepsilon\cap \Delta^r_\Delta\times\partial P=\emptyset,
$$ 
and, similar to the proof of Theorem~\ref{tverberg-rays}, $\gfree(Z_\varepsilon)\ge c + 1$. 

Suppose that $\gfree(Z)\le c$, then its open cover by $c$ inessential sets should be an open cover for $Z_\varepsilon$, for small enough $\varepsilon$. Hence, $\gfree(Z_\varepsilon)\le c$, that is not true. Therefore, $\gfree(Z)\ge c + 1$, and the end of the reasoning is the same as in the proof of Theorem~\ref{tverberg-rays}.


\begin{thebibliography}{99}

\bibitem{adw1984}
R.~Aharoni, P.~Duchet, B.~Wajnryb. Successive projections on hyperplanes. // J. Math. Anal. Appl., 103, 1984, 134--138.

\bibitem{bss1981}
I.~B\'ar\'any, S.B.~Shlosman, S.~Sz\"ucz. On a topological generalization of a theorem of Tverberg. // J. London Math. Soc., II. Ser., 23, 1981, 158--164.

\bibitem{bart1993}
T.~Bartsch. Topological methods for variational problems with symmetries. Berlin-Heidelberg: Springer-Verlag, 1993.

\bibitem{eck1993}
J.~Eckhoff. Helly, Radon, and Carath\'eodory type theorems. // Handbook of Convex Geometry, ed. by P.M.~Gruber and J.M.~Willis, North-Holland, Amsterdam, 1993, 389--448.

\bibitem{fhp2009}
R.~Fulek, A.F.~Holmsen, J.~Pach. Intersecting convex sets by rays. // Discrete and Computational Geometry, 42(3), 2009, 343--358.

\bibitem{grun1960}
B.~Gr\"unbaum. Partitions of mass-distributions and of convex bodies by hyperplanes. // Pacific J. Math., 10, 1960, 1257--1261.

\bibitem{hatcher2002}
A.~Hatcher. Algebraic Topology. Cambridge University Press, 2002.

\bibitem{helly1923}
E.~Helly. \"Uber Mengen konvexer K\"orper mit gemeinschaftlichen Punkten. // Jber Deutsch. Math. Verein., 32, 1923, 175--176.

\bibitem{hsiang1975}
Wu Yi Hsiang. Cohomology theory of topological transformation groups. Springer Verlag, 1975.

\bibitem{kar2008dcpt}
R.N.~Karasev. Dual theorems on central points and their generalizations. // Sbornik: Mathematics, 199(10), 2008, 1459--1479.

\bibitem{kar2009}
R.N.~Karasev. Analogues of the central point theorem for families with $d$-intersection property in $\mathbb R^d$. // \href{http://arxiv.org/abs/0906.2262}{arXiv:0906.2262v1}, 2009.

\bibitem{kar2009bu}
R.N.~Karasev. Theorems of Borsuk-Ulam type for flats and common transversals (In Russian). // Sbornik: Mathematics, 200(10), 2009, 39--58; translated in \href{http://arxiv.org/abs/0905.2747v1}{arXiv:0905.2747v1}.

\bibitem{kr1952}
M.A.~Krasnosel'skii. On the estimation of the number of critical points of functionals (In Russian). // Uspehi Mat. Nauk, 7(2), 1952, 157--164.

\bibitem{mishch1998}
G.~Luke, A.S.~Mishchenko. Vector bundles and their applications. Springer Verlag, 1998.

\bibitem{mat2003}
J.~Matou\v{s}ek. Using the Borsuk-Ulam theorem. Berlin-Heidelberg, Springer Verlag, 2003.

\bibitem{mcc2001}
J.~McCleary. A user's guide to spectral sequences. Cambridge University Press, 2001.

\bibitem{milsta1974}
J.~Milnor, J.~Stasheff. Characteristic classes. Princeton University Press, 1974.

\bibitem{neumann1945}
B.H.~Neumann. On an invariant of plane regions and mass distributions. // J. London Math. Soc., 20, 1945, 226--237. 

\bibitem{rado1946}
R.~Rado. A theorem on general measure. // J. London Math. Soc., 21, 1946, 291--300.

\bibitem{ruhu1999}
P.J.~Rousseeuw, M.~Hubert. Depth in an arrangement of hyperplanes. // Discrete and Computational Geometry, 22, 1999, 167--176.

\bibitem{schw1957}
A.S.~Schwartz. Some estimates of the genus of a topological space in the sense of Krasnosel'skii. (In Russian) // Uspehi Mat. Nauk, 12:4(76), 1957, 209--214. 

\bibitem{schw1966}
A.S.~Schwartz. The genus of a fibre space. // Trudy Moskov. Mat. Obsc., 11, 1962, 99--126; translation in Amer. Math. Soc. Trans., 55, 1966, 49--140.

\bibitem{tver1966}
H.~Tverberg. A generalization of Radon's theorem. // J. London Math. Soc., 41, 1966, 123--128.

\bibitem{vol1996}
A.Yu.~Volovikov. On a topological generalization of the Tverberg theorem. // Mathematical Notes, 59(3), 1996, 324--326.

\bibitem{vol2007}
A.Yu.~Volovikov. On the Cohen-Lusk theorem (In  Russian). // Fundamental and Applied Mathematics, 13(8), 2007, 61--67.


\end{thebibliography}
\end{document}